\documentclass[11pt,letterpaper]{amsart}

\usepackage{amsmath,amsfonts,amssymb,amsthm,url,array}
\usepackage{amsthm}
\usepackage{mathrsfs}
\usepackage{url}
\usepackage{color}

\usepackage{hyperref}
\usepackage[T1]{fontenc}

\usepackage[a4paper,top=2cm,bottom=2cm,left=2cm,right=2cm,marginparwidth=1.75cm]{geometry}

\usepackage{graphicx}

\newtheorem{thm}{Theorem}[section]

\newtheorem{cor}[thm]{Corollary}
\newtheorem{lem}[thm]{Lemma}
\newtheorem{rmk}[thm]{Remark}

\numberwithin{equation}{section}

\newcommand{\Z}{{\mathbb{Z}}}
\newcommand{\R}{{\mathbb{R}}}
\newcommand{\Q}{{\mathbb{Q}}}
\newcommand{\co}{{\mathcal{O}}}
\newcommand{\Tr}{\mathrm{Tr}}
\newcommand{\Nr}{\mathrm{N}}

\newcommand{\eq}{\equiv}
\newcommand{\rom}[1]{\uppercase\expandafter{\romannumeral #1\relax}}

\newcommand{\Vol}{{\rm Vol}}

\allowdisplaybreaks

\title[Real quadratic fields with a universal form]{Real quadratic fields with a universal quadratic form \\ of given rank have density zero}

\author{V\' \i t\v ezslav Kala}
\address{Charles University, Faculty of Mathematics and Physics, Department of Algebra, Sokolov\-sk\' a 83, 18600 Praha~8, Czech Republic}
\email{{vitezslav.kala@matfyz.cuni.cz}}

\author{Pavlo Yatsyna}
\address{Charles University, Faculty of Mathematics and Physics, Department of Algebra, Sokolov\-sk\' a 83, 18600 Praha~8, Czech Republic}
\email{{p.yatsyna@matfyz.cuni.cz}}

\author{B\l{}a\.zej \.Zmija}
\address{Charles University, Faculty of Mathematics and Physics, Department of Algebra, Sokolov\-sk\' a 83, 18600 Praha~8, Czech Republic}
\address{Institute of Mathematics of the Polish Academy of Sciences, \'{S}niadeckich 8, 00-656 Warsaw, Poland}
\email{blazej.zmija@matfyz.cuni.cz}

\thanks{V.K. and B.\.Z. were supported by Czech Science Foundation (GA\v CR) grant 21-00420M. P.Y. was supported by Research Council of Finland (grant 351271, PI C. Hollanti) and by  Charles University programmes PRIMUS/24/SCI/010 and UNCE/24/SCI/022}
\thanks{To appear in the American Journal of Mathematics}
\keywords{universal quadratic form, quadratic lattice, real quadratic number field, continued fraction}
\subjclass[2020]{11A55, 11E12, 11E20, 11H55, 11R11, 11R80}

\begin{document}

\begin{abstract} 
	We prove an explicit upper bound on the number of real quadratic fields that admit a universal quadratic form of a given rank, thus establishing a density zero statement. More generally, we obtain such a result for totally positive definite quadratic lattices that represent all the multiples of a given rational integer.
	Our main tools are short vectors in quadratic lattices combined with an estimate for the number of periodic continued fractions with bounded coefficients.
\end{abstract}

\maketitle

\section{Introduction}

Among the most widely recognized results in number theory are Fermat's two-square, Legendre's three-square, and Lagrange's four-square theorems. From Gauss to Minkowski, Hilbert, Ramanujan, and Siegel, the study of integers represented by the sum of squares and other quadratic forms has made tremendous advances. 
The most recent breakthroughs occurred in the form of the so-called ``number theorems'' by Conway--Schneeberger \cite{Co}, Bhargava--Hanke \cite{BH}, and Rouse \cite{Ro}, which gave sufficient and necessary conditions for a positive definite quadratic form to represent all positive integers or other special infinite sets. 

Similar questions are quite difficult for quadratic forms over the rings of integers $\co_H$ of number fields $H$. When the form is indefinite (or when there is at least one complex, non-real embedding of $H$), then a lot of information is given by a local-global principle \cite{HHX, HSX, XZ}. The situation is markedly different for positive definite quadratic forms over totally real number fields. While there is an ``asymptotic local-global principle'' \cite{HKK} that holds for all elements of sufficiently large norm (provided that the quadratic form has rank at least five), determining the representability of small elements is very hard, and so is studying universal forms, i.e., those that represent all the totally positive integers.

First, note that already Siegel \cite{Si2} showed that the sum of any number of squares is universal only over $\Q$ and $\Q(\sqrt 5)$ (where three squares suffice \cite{Maa}), and thus one needs to consider more general forms.

To be more precise, let us talk about totally positive definite \textit{quadratic $\co_H$-lattices} $(\Lambda,Q)$ over a totally real number field $H$, i.e., finitely generated $\co_H$-modules $\Lambda$ equipped with a quadratic form $Q$ such that all the values $Q(v)$ for $0\neq v\in \Lambda$ are totally positive elements of $\co_H$. Such a lattice is \textit{universal} if it represents all the totally positive integers; we will also often assume that $\Lambda$ is \textit{classical} in the sense that all the values of the associated bilinear form lie in $\co_H$
(for more precise definitions and further information, see the beginning of Section \ref{sec: large ranks}). As a universal classical lattice exists over every $H$, we can denote by $R_{class}(H)$ its smallest possible rank over $H$ (and by $R(H)$ the smallest rank without the classical assumption).

In contrast with quadratic forms over rational integers $\Z$, there exist universal quadratic lattices of rank three over some real quadratic fields -- they were all characterized first in the classical case \cite{CKR}, and then in general \cite{KK+}. However, Kitaoka formulated 
the influential conjecture that there are only finitely many number fields that admit a ternary universal lattice; in fact, 
only two fields in higher degrees are known [Krásenský--Scharlau, in preparation]. Despite some progress \cite{EK, KY2}, the conjecture remains open.
Over real quadratic fields $H=\Q(\sqrt D)$, even more is known. Namely, only finitely many of them have $R(\Q(\sqrt D))\leq 7$ \cite{KKP}, but there is a universal form of rank 8 whenever $D={n^2-1}$ is squarefree \cite{Ki}. 

The works of Blomer and Kala \cite{BK,BK2,Ka} brought a breakthrough in the area by establishing that the ranks $R(\Q(\sqrt D))$ can be arbitrarily large. Their results were based on the observation that it is hard for a quadratic lattice to represent certain \textit{indecomposable} elements, and so their number can be used to obtain lower bounds for $R(\Q(\sqrt D))$.
These indecomposables are in turn connected to the continued fraction of $\sqrt D$ \cite{DS}, and so by controlling the continued fraction, one can get large ranks of universal forms. However, these arguments were rather delicate, presenting a major limitation to their applicability. In particular, they yielded only very sparse sets of $D$ with large $R(\Q(\sqrt D))$; some of the obstacles were closely related to the problem of existence of real quadratic fields with small class number (for if the class number is large, then it is hard to argue that there exist suitable elements of small norm).

\medskip

The current paper overcomes these difficulties and shows that, in fact, the ranks $R(\Q(\sqrt D))$ are typically large.

\begin{thm}\label{th1.1}
	Let $\varepsilon >0$.
	For almost all squarefree $D>0$, we have that
	$$R_{class}(\Q(\sqrt D))\geq D^{\frac{1}{12} -\varepsilon}\text{\ \ \ and\ \ \ }R(\Q(\sqrt D))\geq D^{\frac{1}{24} -\varepsilon}.$$
\end{thm}

By ``almost all'' we mean that such $D$ have (natural) density 1 among the set of all squarefree $D>0$. 

\medskip

In fact, we work in the more general setting of representing all multiples of a fixed positive rational integer~$m$. We say that a  quadratic lattice $(\Lambda,Q)$ over a totally real number field $H$ is \textit{$m\co_H$-universal} if $Q$ represents all the elements of 
$m\co_H^{+}$, i.e., all the totally positive multiples of $m$.

Not only do  $m\co_H $-universal lattices present a natural generalization, they also provide  an elegant way of dealing with non-classical universal lattices $(\Lambda,Q)$, for then $(\Lambda,2Q)$ is \textit{classical} and $2\co_H$-universal -- and it is usually more convenient to work under the classical assumption, as we do in this paper.

As our main result, we thus show an explicit upper bound on the number of real quadratic fields that admit an 
$m\co_H $-universal lattice of given rank $R$.

\begin{thm}\label{MAIN}
	Let $R,m,X$ be positive integers. Denote
	\begin{align*}
	&\mathcal{D}(R,m,X):=\\
	&\ \#\left\{ \text{squarefree }  D\leq X \mid \exists \text{ $m\co_{\Q(\sqrt{D})}$-universal classical lattice of rank } R  \right\}.
	\end{align*}
	Then for all sufficiently large $X$ we have
	\begin{align*}
		\mathcal{D}(R,m,X)< A(R,m)\cdot X^{7/8} (\log X)^{3/2}
	\end{align*}
	for some constant $A(R,m)$ that is given explicitly in Theorem $\ref{th:main}$.
\end{thm}

More precise versions of Theorems \ref{th1.1} and \ref{MAIN} are proved in the last section as Corollary \ref{cor:main} and Theorem~\ref{th:main}.

\medskip

Let us describe the idea of the proof and at the same time present the structure of the paper. Kala and Tinkov\' a \cite{KT} proved that the ranks of universal forms over $\Q (\sqrt{D})$ are closely related to the size of the largest (odd indexed) coefficient in the continued fraction expansion of $\xi_{D}$, the generator of the ring of integers of $\Q (\sqrt{D})$. Therefore, in Section \ref{sec:cont frac} we review the necessary preliminaries concerning continued fractions and then focus on studying the number of integers $D\leq X$ such that all the continued fraction coefficients (at odd indices) of $\xi_{D}$ are bounded. In fact, our method applies to a more general setting, and so in Theorem \ref{ThmBound} we study the set of the form
\begin{align*}
\#\big\{\ 1\leq D\leq X \ \big|\  f_{D}\bmod 1 = [0;u_{1},u_{2},\ldots ],\ u_{j}\leq\varphi (j) \textrm{ for all } j \ \big\},
\end{align*}
where $(f_{D})_{D=1}^{\infty}$ is an increasing sequence such that the difference sequence $(f_{D+1}-f_{D})_{D=1}^{\infty}$ is decreasing 
and $\varphi :\mathbb{N}\to \mathbb{N}\cup\{\infty\}$ satisfies $\sum_{n=1}^{\infty}\frac{1}{\varphi (n)}=\infty$. The specialization of this theorem to the case that we need is provided in Corollary \ref{CorBound}.

At the beginning of the proof of Theorem \ref{ThmBound} we switch our attention from sequences to subsets of the interval $[0,1)$, especially sets of all the real numbers within $[0,1)$ that have the first $n$ continued fraction coefficients fixed and equal to some prescribed numbers. This is done in Steps I and II of the proof and the main idea of this part comes from \cite[Chapter III]{Khinchin}. We then show in Step III a quantitative version of some equidistribution relation satisfied by the sequence $(f_{D})_{D=1}^{\infty}$. Here we make use of Erd\H{o}s--Tur\'{a}n inequality \cite[Corollary 1.1]{Montgomery} and some bounds on trigonometric sums. In step IV we use all the previously obtained results to conclude the proof.

Having Theorem \ref{ThmBound} proved, we want to find a bound for the largest odd indexed coefficient $u$ of $\xi_{D}$ for $D$ such that there exists an $m\co_{\Q(\sqrt{D})} $-universal lattice of rank $r$ over $\Q(\sqrt{D})$. For this purpose, we estimate the numbers $N(n)$ of vectors of norm $n$ in a quadratic $\Z$-lattice $(\Lambda,Q)$ with Gram matrix $G$ in Section \ref{sec:theta est}. 
In Theorem \ref{thm:short vectors} we thus obtain the following estimates:
\begin{align*}
N(1)\leq 2r;\ \ N(2)\leq \max\{480, 2r(r-1)\};\\
N(n)\leq \frac{\pi^{\frac{r}{2}}}{\Gamma\left(\frac{r}{2}+1\right)} \frac{n^{\frac{r}{2}}}{\sqrt{\det G}} + O(n^{\frac{r-1}{2}})\text{ for }n\geq 3.
\end{align*}
Such estimates are known \cite[Lemma 4.1]{Bl}, \cite[Theorem 20.9]{IwanKowal} (cf. also the very recent result of Regev and Stephens-Davidowitz \cite[Theorem 1.1]{RS}), but we make them fully explicit. 
Using Theorem \ref{thm:short vectors} and the ideas from \cite{KT}, we then easily obtain the required bound on $u$ in Theorem \ref{thm:large ranks}. This together with Theorem \ref{ThmBound} is enough to prove 
our main theorems as Theorem~\ref{th:main} and Corollary \ref{cor:main}, which we do at the end of the paper.

\medskip

Our results leave open the tantalizing question of the behavior of the ranks $R(H)$ for number fields $H$ of higher degree.
The few available results \cite{Ka2, KS, KT, Ti, Ya} suggest that the ranks are perhaps also typically large, but this is probably far out of reach, as the potential generalizations of continued fractions do not seem to be available for such an application -- and even Kitaoka's conjecture concerning $R(H)=3$ remains unproven.
One fascinating exception is the very recent work of Man \cite{Man} who used our results in his proofs of analogues of Theorems \ref{th1.1} and \ref{MAIN} for multiquadratic fields.

\section*{Acknowledgments}

We thank Miko\l aj Fr\k aczyk, Siu Hang Man, and Dayoon Park for interesting and helpful discussions about the paper.

\section{Continued fraction coefficients}\label{sec:cont frac}

At the beginning of this section let us recall some basic facts about continued fractions. Every irrational number $x$ can be expressed in the following way:
\begin{align*}
x=u_{0} + \cfrac{1}{u_{1} + \cfrac{1}{ u_{2} + \cfrac{1}{ u_{3} + \cfrac{1}{\ddots}}}} =: [u_{0};u_{1},u_{2},\ldots ],
\end{align*}
where $u_{0}:=\lfloor x\rfloor\in \mathbb{Z}$ and $u_{j}\in\mathbb{N}$ for $j\geq 1$ (we denote the set of positive integers by $\mathbb{N}$). The above expression is called {\it continued fraction expansion} and the numbers $u_{j}$ are the \textit{coefficients} of the expansion. A continued fraction is {\it periodic} if there are $N$ and $l$ such that $u_{n}=u_{n+l}$ for all $n\geq N$. In this situation we write
\begin{align*}
x = [u_{0};u_{1}\ldots ,u_{N-1},\overline{u_{N},\ldots ,u_{N+l-1}}]
\end{align*}
for simplicity. It is well known that a number $x$ has periodic continued fraction expansion if and only if $x$ is a quadratic irrational.

If $x=[u_{0};u_{1},u_{2},\ldots ]$, then for every positive integer $n$ we can consider the {\it convergent}
\begin{align*}
[u_{0};u_{1},u_{2},\ldots ,u_{n}] =: \frac{p_{n}}{q_{n}},
\end{align*}
where $p_{n}$ and $q_{n}$ are coprime positive integers. Moreover, they satisfy the following recurrence relations:
\begin{align*}
\left\{\begin{array}{ll}
p_{-2}=0, & \\
p_{-1}=1, & \\
p_{n}=u_{n}p_{n-1}+p_{n-2}, & n\geq 0,
\end{array}\right. 
\hspace{1cm}
\left\{\begin{array}{ll}
q_{-2}=1, & \\
q_{-1}=0, & \\
q_{n}=u_{n}q_{n-1}+q_{n-2}, & n\geq 0.
\end{array}\right.
\end{align*}
We also have that for every $n$,
\begin{align*}
q_{n}p_{n-1}-q_{n-1}p_{n}=(-1)^{n}.
\end{align*}

The proofs of the statements above can be found in any survey concerning continued fractions, for example in the book \cite{Khinchin}. We will use them throughout the paper without reference.

Let us reintroduce one more piece of notation. For a fixed squarefree $D\in \Z_{>1}$, we denote the generator of the ring of integers of $\Q (\sqrt{D})$ by
\begin{equation*}
\xi_D:=\begin{cases}\sqrt{D} & \text{when }D \eq 2,3 \pmod 4, \\
\frac{1+\sqrt{D}}{2} & \text{when }D \eq 1 \pmod 4. \end{cases}
\end{equation*}
In Section \ref{sec: large ranks} we show a connection between the continued fraction expansions of the numbers $\xi_{D}$ and the ranks of universal quadratic forms over $\Q(\sqrt{D})$.

Let us move on to the main part of this section. Our aim is to present a general theorem (Theorem \ref{ThmBound}) that provides an explicit upper bound for the numbers of terms (up to some $X$) of a given sequence that have bounded coefficients in their continued fraction expansion. Later, in Corollary \ref{CorBound}, we apply this to the case of the numbers $\xi_{D}$.

\medskip

In this paper we mainly focus on the case of the sequence $(\xi_{D})_{4\nmid D}$. Therefore, at first, let us narrow our considerations down to the case of square roots and present a rough heuristic argument explaining why we could expect for every fixed number $B\geq 1$ that the set of numbers $D$ such that all the coefficients $u_{2i+1}$ in the continued fraction expansions of $\sqrt{D}$ are bounded by $B$, has natural density~$0$.

Every set of the form
\begin{align*}
E\begin{pmatrix}
t_{1} & t_{2} & \ldots & t_{n} \\ k_{1} & k_{2} & \ldots & k_{n}
\end{pmatrix}
:= \{\ \alpha\in [0,1)\ |\ \alpha=[0;u_{1},u_{2},\ldots ],\ u_{t_{j}}=k_{j} \ \}
\end{align*}
is an interval or a countable sum of intervals with the endpoints in $\Q$ (see \cite[pp. 57--58]{Khinchin}; we will use a more precise description later, see Lemma \ref{LemProp1}  below). In particular, every set
\begin{align*}
E\begin{pmatrix}
1 & 2 & \ldots & n \\ k_{1} & k_{2} & \ldots & k_{n}
\end{pmatrix}
\end{align*}
is an interval that we call an interval of \textit{rank} $n$. One can check that for example
\begin{align*}
E\begin{pmatrix}
1 \\ k
\end{pmatrix}=\left[\frac{1}{k+1},\frac{1}{k}\right),
\end{align*}
see also \cite[Chapter III]{Khinchin}. If $D$ is such that $u_{1}\leq B$, then $\sqrt{D}-\lfloor\sqrt{D}\rfloor$ belongs to the union of the sets $E\begin{pmatrix}
1 \\ k
\end{pmatrix}$ for $1\leq k\leq B$, that is, to the interval $[1/(B+1),1)$. The probability of such an event should be therefore close to $1-1/(B+1)$. We could also expect that the probability that $u_{2i+1}\leq B$ is similar for every $i$, and so close to $1-1/(B+1)$, too. Moreover, it seems reasonable to suppose that such events are independent. Hence, the probability that $\sqrt{D}$ has all the coefficients $u_{2i+1}$ bounded by $B$ for $1\leq i\leq n$ should be close to
\begin{align*}
\left(1-\frac{1}{B+1}\right)^{n}
\end{align*}
which tends to $0$ as $n\to\infty$.

\medskip

The above heuristic argument can easily be generalized. We can expect that for a given function $\varphi :\mathbb{N}\to \mathbb{N}\cup\{\infty\}$, the more general set 
\begin{align*}
\big\{\ 1\leq D\leq X \ \big|\  \sqrt{D}-\lfloor\sqrt{D}\rfloor = [0;u_{1},u_{2},\ldots ],\ u_{n}\leq\varphi (n) \textrm{ for all } n \ \big\},
\end{align*}
has density zero if we only assume that $\varphi (n)\neq \infty$ for sufficiently many $n$. In fact, we shall prove a result that can be applied to a more general class of sequences than only $(\sqrt{D}-\lfloor\sqrt{D}\rfloor)_{D=1}^{\infty}$. This is stated more precisely in the following theorem.

In order to simplify the notation let us denote for every real number $r$ its fractional part as $r\bmod{1}$. That is, 
\begin{align*}
r\bmod{1} := r-\lfloor r\rfloor\in [0,1).
\end{align*} 

\begin{thm}\label{ThmBound}
Let $(f_{D})_{D=1}^{\infty}$ be an increasing sequence such that the sequence $(f_{D+1}-f_{D})_{D=1}^{\infty}$ is decreasing. 

Let $\varphi :\mathbb{N}\to \mathbb{N}\cup\{\infty\}$ be such that $\varphi (1)\neq\infty$ and $\sum_{n=1}^{\infty}\frac{1}{\varphi (n)}=\infty$. 

Let $(a_{n})_{n=1}^{\infty}$ be an increasing sequence of positive integers such that $\varphi (a)\neq\infty$ if and only if there exists $n$ satisfying $a=a_{n}$. 

Then for every $X$, $L$, $n>1$, the quantity
\begin{align*}
\#\big\{\ 1\leq D\leq X \ \big|\  f_{D}\bmod 1 = [0;u_{1},u_{2},\ldots ],\ u_{j}\leq\varphi (j) \textrm{ for all } j \ \big\}
\end{align*}
is bounded from above by
\begin{align*}
X & \left[\prod_{j=1}^{n}\left(1-\frac{1}{3(\varphi (j)+2)}\right) +\frac{2(n-1)}{L}\right]+ \\
 & \hspace{0.5cm} 
 + \left((3\pi+1)X^{1/2}f_{X+1}^{1/2}+\frac{\pi}{2}\cdot \frac{1}{f_{X+2}-f_{X+1}}\right) \\
 & \hspace{0.5cm} \cdot \left[ \sum_{j=1}^{n}(n-j+1)\varphi (a_{j}) + \left(\sum_{j=1}^{n}a_{j}-\frac{n(n+1)}{2}\right)L\right].
\end{align*}
\end{thm}

\begin{rmk}
The assumptions that the sequence $(f_{D})_{D=1}^{\infty}$ is  increasing and $(f_{D+1}-f_{D})_{D=1}^{\infty}$ is decreasing can be relaxed, see Step III of the proof below. We skipped this more general form to keep the bound less complicated.
\end{rmk}

In fact, in order to get the density zero result in the case of the sequence $\big(\sqrt{D}-\lfloor\sqrt{D}\rfloor\big)_{D=1}^{\infty}$, it is enough to use the fact that this sequence is equidistributed. This means that for every set $I$ which is a finite sum of intervals, we have
\begin{align}\label{EquiLim}
\lim_{X\to\infty}\frac{\#\big\{\ 1\leq D\leq X\ \big|\ \sqrt{D}-\lfloor\sqrt{D}\rfloor\in I \ \big\}}{X}=\mu (I),
\end{align}
where, as in the rest of the paper, $\mu$ denotes the Lebesgue measure. Thanks to this, it is enough to study the set
\begin{align*}
F_{\varphi}:= \{\ \alpha=[0; u_1, u_2, \dots] \in [0,1)\ |\ \ u_{n}\leq \varphi (n)  \textrm{ for all } n \ \}.
\end{align*}
This provides a motivation to focus on the set $F_{\varphi}$ and we do so in our proof.

\medskip

We divide the proof of Theorem \ref{ThmBound} into the following four steps.
\begin{enumerate}
\item[{\bf I.}]  We show that the set $F_{\varphi}$ is contained in sets with a small Lebesgue measure (depending on the parameter $n$).
\item[{\bf II.}] We modify the sets from Step I to be the unions of finitely many intervals and still have a small Lebesgue measure (both the number of intervals and the Lebesgue measure of the new sets will depend on the parameter $L$).
\item[{\bf III.}] We prove a quantitative version of \eqref{EquiLim} for $f_D$.
\item[{\bf IV.}] We use the results from the previous steps to prove the main statement.
\end{enumerate}

\subsection*{Step I} In the first step our aim is to prove the following theorem.

\begin{thm}\label{StepI}
Let
\begin{align*}
F_{\varphi}:= \{\ \alpha=[0; u_1, u_2, \dots]\in [0,1)\ |\ \ u_{n}\leq \varphi (n)    \textrm{ for all } n \ \}.
\end{align*}
Then for every $n\in\mathbb{N}$ there exists a set $F_{\varphi}^{(a_{n})}$ such that $F_{\varphi}\subseteq F_{\varphi}^{(a_{n})}$ and
\begin{align*}
\mu\left(F_{\varphi}^{(a_{n})}\right)<\prod_{j=1}^{n}\left(1-\frac{1}{3(\varphi (j)+2)}\right).
\end{align*}
\end{thm}
\begin{proof}
The idea of the proof comes from \cite[Chapter III]{Khinchin}. We will need the following fact that is proved as part of the proof of \cite[Theorem 29]{Khinchin}.

\begin{lem}\label{LemIneq}
Let $J_{n}$ be an interval of rank $n$ and for every $k$ let $J_{n+1}^{(k)}$ be the subinterval of $J_{n}$ of numbers satisfying $u_{n+1}=k$. Then for every $N\in\mathbb{N}_{0}$:
\begin{align*}
\mu\left(\bigcup_{k=N+1}^{\infty}J_{n+1}^{(k)}\right)>\frac{1}{3(N+2)}\mu (J_{n}).
\end{align*}
\end{lem}

For an interval $J_{M}$ of rank $M$ let $J_{M,t}^{(k_{1},\ldots ,k_{t})}$ denote the subinterval of $J_{M}$ such that $u_{M+j}=k_{j}$ for every $1\leq j\leq t$. Let $M=a_{n}$ and $N=a_{n+1}$. Let us denote $L:=N-M$. Then by Lemma \ref{LemIneq} we get:
\begin{align*}
 &\ \mu  \left(\bigcup_{k_{1}=1}^{\infty} \cdots \bigcup_{k_{L-1}=1}^{\infty}\bigcup_{k_{L}=1}^{\varphi (n+1)} J_{M,L}^{(k_{1},\ldots ,k_{L})}\right)  \\
 = &\ \mu\left(\bigcup_{k_{1}=1}^{\infty}\cdots \bigcup_{k_{L}=1}^{\infty} J_{M,L}^{(k_{1},\ldots ,k_{L})}\right) - \mu\left(\bigcup_{k_{1}=1}^{\infty}\cdots \bigcup_{k_{L-1}=1}^{\infty}\bigcup_{k_{L}=\varphi (n+1)+1}^{\infty} J_{M,L}^{(k_{1},\ldots ,k_{L})}\right) \\
 = &\ \mu (J_{M})-\sum_{k_{1}=1}^{\infty}\sum_{k_{2}=1}^{\infty}\cdots\sum_{k_{L-1}=1}^{\infty}\mu\left(\bigcup_{k_{L}=\varphi (n+1)+1}^{\infty}J_{M,L}^{(k_{1},\ldots ,k_{L})}\right) \\
= &\ \mu (J_{M})-\sum_{k_{1}=1}^{\infty}\sum_{k_{2}=1}^{\infty}\cdots\sum_{k_{L-1}=1}^{\infty}\mu\left(\bigcup_{k_{L}=\varphi (n+1)+1}^{\infty}\left( J_{M,L-1}^{(k_{1},\ldots ,k_{L-1})}\right)^{(k_{L})}_{1}\right) \\
< &\ \mu (J_{M})-\frac{1}{3(\varphi(n+1)+2)}\sum_{k_{1}=1}^{\infty}\sum_{k_{2}=1}^{\infty}\cdots\sum_{k_{L-1}=1}^{\infty}\mu\left(J_{M,L-1}^{(k_{1},\ldots ,k_{L-1})}\right) \\
= &\ \mu(J_{M})-\frac{1}{3(\varphi(n+1)+2)}\mu (J_{M})=\left(1-\frac{1}{3(\varphi(n+1)+2)}\right)\mu(J_{M}).
\end{align*}
Denote $\tau_{n}:=1-\frac{1}{3(\varphi(n)+2)}$.

For every $n\in\mathbb{N}$ let:
\begin{align*}
F_{\varphi}^{(a_{n})}:= &\ \{\ \alpha\in [0,1)\ |\ u_{k}\leq \varphi (k) \textrm{ for all } k\leq a_{n}\ \}, \\
\mathcal{J}_{\varphi}^{(a_{n})}:= &\ \{\ J\subseteq [0,1)\ |\ J \textrm{ is an interval of rank } a_{n} \textrm{ with } J\cap F_{\varphi}^{(a_{n})}\neq\emptyset \ \}.
\end{align*}
From the definitions of the above sets we get for every $n$ and for every interval $J_{M}$ of rank $M=a_{n}$:
\begin{enumerate}
\item $F_{\varphi}\subseteq F_{\varphi}^{(a_{n})}$,
\item $F_{\varphi}^{(a_{n+1})}\cap J_{M}\subseteq \bigcup_{k_{1}=1}^{\infty}\cdots \bigcup_{k_{L-1}=1}^{\infty}\bigcup_{k_{L}=1}^{\varphi (n+1)}J_{M,L}^{(k_{1},\ldots ,k_{L})}$,
\item $\bigcup_{J\in\mathcal{J}_{\varphi}^{(a_{n})}}J=F_{\varphi}^{(a_{n})}$.
\end{enumerate}

The above properties imply for every $n$:
\begin{align*}
\mu (F_{\varphi}^{(a_{n+1})})= &\ \sum_{J_{M}\in\mathcal{J}_{\varphi}^{(a_{n})}}\mu (F_{M}^{(a_{n+1})}\cap J_{M})\\
\leq &\  \sum_{J_{M}\in\mathcal{J}_{\varphi}^{(a_{n})}}\mu \left(\bigcup_{k_{1}=1}^{\infty}\cdots \bigcup_{k_{L-1}=1}^{\infty}\bigcup_{k_{L}=1}^{\varphi (n+1)}J_{M,L}^{(k_{1},\ldots ,k_{L})}\right) \\
 < &\ \sum_{J_{M}\in\mathcal{J}_{\varphi}^{(a_{N})}} \tau_{n+1} \mu(J_{M})=\tau_{n+1} \mu\bigg(\bigcup_{J_{M}\in\mathcal{J}_{\varphi}^{(a_{n})}}J_{M}\bigg)\\
 = &\  \tau_{n+1} \mu (F_{\varphi}^{(a_{n})})<\tau_{n+1}\tau_{n}\mu(F_{\varphi}^{(a_{n-1})}) \\
 < &\ \ldots <\tau_{n+1}\tau_{n}\cdots\tau_{1} \mu(F_{\varphi}^{(a_{1})})\leq \prod_{j=1}^{n+1}\tau_{j}.
\end{align*}
This finishes the proof of the first step.
\end{proof}

\subsection*{Step II} The main result of this step is the following.

\begin{thm}\label{StepII}
For every positive integers $n$ and $L$ there exists a set $F_{\varphi}^{(a_{n},L)}$ such that:
\begin{enumerate}
\item $F_{\varphi}\subseteq F_{\varphi}^{(a_{n},L)}$.
\item $F_{\varphi}^{(a_{n},L)}$ is a union of at most
\begin{align*}
\sum_{j=1}^{n}(n-j+1)\varphi (a_{j}) + \left(\sum_{t=1}^{n}a_{t}-\frac{n(n+1)}{2}\right)L
\end{align*}
disjoint intervals.
\item We have
\begin{align*}
\mu \left(F_{\varphi}^{(a_{n},L)}\right) <\prod_{j=1}^{n}\left(1-\frac{1}{3(\varphi (j)+2)}\right)+\frac{5(n-1)}{3L}.
\end{align*}
\end{enumerate}
\end{thm}

\begin{proof}
In this part, we will need the following more precise description of intervals of a fixed rank $N$, that is proved at the beginning of Chapter III in \cite{Khinchin}.

\begin{lem}\label{LemProp1}
Let $(k_{1},\ldots ,k_{N})$ be a sequence of natural numbers. The endpoints of the interval 
\begin{align*}
E\begin{pmatrix}
1 & 2 & \ldots & N \\ k_{1} & k_{2} & \ldots & k_{N}
\end{pmatrix}
\end{align*}
are $\frac{p_{N}}{q_{N}}$ and $\frac{p_{N}+p_{N-1}}{q_{N}+q_{N-1}}$, where $\frac{p_{j}}{q_{j}}$ denotes the $j$th convergent of the continued fraction $[k_{1};k_{2},\ldots ,k_{N}]$.
\end{lem}

For real numbers $a$ and $b$ let us denote
\begin{align*}
[a,b]^{\star}:=\left\{\begin{array}{ll}
[a,b], & \textrm{if } a\leq b, \\
\left[b,a\right], & \textrm{if } b\leq a.
\end{array}\right.
\end{align*}

\begin{lem}\label{LemProp2}
Let $(k_{1},\ldots ,k_{N})$ be a sequence of natural numbers such that $k_{N}\geq L+1$. Then
\begin{align*}
E\begin{pmatrix}
1 & 2 & \ldots & N \\ k_{1} & k_{2} & \ldots & k_{N}
\end{pmatrix}
\subseteq \left[\frac{p_{N-1}+\frac{p_{N-2}}{L+1}}{q_{N-1}+\frac{q_{N-2}}{L+1}},\frac{p_{N-1}}{q_{N-1}}\right]^{\star},
\end{align*}
where the numbers $p_{j}$ and $q_{j}$ are the same as in Lemma $\ref{LemProp1}$.
\end{lem}
\begin{proof}
At first let us assume that $N$ is even. From Lemma \ref{LemProp1} and the recurrence relations satisfied by the numbers $p_{j}$ and $q_{j}$ we get:
\begin{align*}
E\begin{pmatrix}
1 & 2 & \ldots & N \\ k_{1} & k_{2} & \ldots & k_{N}
\end{pmatrix}
&\subseteq
\left[\frac{p_{N}}{q_{N}},\frac{p_{N}+p_{N-1}}{q_{N}+q_{N-1}}\right]\\
&=\left[\frac{k_{N}p_{N-1}+p_{N-2}}{k_{N}q_{N-1}+q_{N-2}},\frac{k_{N}p_{N-1}+p_{N-1}+p_{N-2}}{k_{N}q_{N-1}+q_{N-1}+q_{N-2}}\right].
\end{align*}

It is easy to check that if $a$, $b$, $c$ and $d$ are such that $ad-bc>0$, then the function $f(x)=\frac{ax+b}{cx+d}$ is increasing. Observe that
\begin{align*}
p_{N-1}q_{N-2}-p_{N-2}q_{N-1}=(-1)^{N}=1>0
\end{align*}
and
\begin{align*}
p_{N-1}(q_{N-1}+q_{N-2})-(p_{N-1}+p_{N-2})q_{N-1}=(-1)^{N}=1>0.
\end{align*}
Therefore, if
\begin{align*}
f_{1}(x):=\frac{xp_{N-1}+p_{N-2}}{xq_{N-1}+q_{N-2}} \hspace{1cm} \textrm{ and } \hspace{1cm} f_{2}(x):=\frac{xp_{N-1}+p_{N-1}+p_{N-2}}{xq_{N-1}+q_{N-1}+q_{N-2}},
\end{align*}
then
\begin{align*}
E\begin{pmatrix}
1 & 2 & \ldots & N \\ k_{1} & k_{2} & \ldots & k_{N}
\end{pmatrix}
&\subseteq
\big[ f_{1}(k_{N}),f_{2}(k_{N})\big]\subseteq \big[f_{1}(L+1),f_{2}(\infty)\big]\\
& = \left[\frac{(L+1)p_{N-1}+p_{N-2}}{(L+1)q_{N-1}+q_{N-2}},\frac{p_{N-1}}{q_{N-1}}\right].
\end{align*}
The result follows. If $N$ is odd, then the proof is analogous. The only difference is that the endpoints of the last interval are reversed and the functions $f_{1}$ and $f_{2}$ are decreasing.
\end{proof}

Let us further denote
\begin{align*}
I(k_{1},\ldots ,k_{N-1}):=\left[\frac{p_{N-1}+\frac{p_{N-2}}{L+1}}{q_{N-1}+\frac{q_{N-2}}{L+1}},\frac{p_{N-1}}{q_{N-1}}\right]^{\star}.
\end{align*}
Note that this set does not depend on $k_{N}$.

We are ready to prove Theorem \ref{StepII}. We can write
\begin{align*}
F_{\varphi}^{(a_{n})}=\bigcup_{k_{1}=1}^{\varphi (1)}\bigcup_{k_{2}=1}^{\varphi (2)}\cdots \bigcup_{k_{a_{n}-1}=1}^{\varphi (a_{n}-1)}\bigcup_{k_{a_{n}}=1}^{\varphi (a_{n})}E\begin{pmatrix}
1 & 2 & \ldots & a_{n} \\ k_{1} & k_{2} & \ldots & k_{a_{n}}
\end{pmatrix}.
\end{align*}
Lemma \ref{LemProp2} implies that for every $1\leq N\leq n$ and for every fixed sequence $(k_{1},\ldots ,k_{a_{N}})$ we have
\begin{align*}
\bigcup_{k_{a_{N}+1}=L+1}^{\varphi (a_{N} +1)} & \bigcup_{k_{a_{N}+2}=1}^{\varphi (a_{N}+2)}\bigcup_{k_{a_{N}+3}=1}^{\varphi (a_{N}+3)}
\cdots \bigcup_{k_{a_{n}-1}=1}^{\varphi (a_{n}-1)}\bigcup_{k_{a_{n}}=1}^{\varphi (a_{n})}E\begin{pmatrix}
1 & 2 & \ldots & 2n+1 \\ k_{1} & k_{2} & \ldots & k_{2n+1}
\end{pmatrix} \\
\subseteq &\ \bigcup_{k_{a_{N}+1}=L+1}^{\infty} E\begin{pmatrix}
1 & 2 & \ldots & a_{N}+1 \\ k_{1} & k_{2} & \ldots & k_{a_{N}+1}
\end{pmatrix}
 \subseteq I(k_{1},\ldots ,k_{a_{N}}).
\end{align*}

The assumption $\varphi (1)\neq\infty$ implies $a_{1}=1$. Therefore,
\begin{align*}
F_{\varphi}^{(a_{n})}\subseteq &\ \bigcup_{k_{1}=1}^{\varphi (1)}\bigcup_{k_{2}=1}^{L}\bigcup_{k_{3}=1}^{L}\cdots \bigcup_{k_{a_{2}-1}=1}^{L}\bigcup_{k_{a_{2}}=1}^{\varphi (a_{2})}\bigcup_{k_{a_{2}+1}=1}^{L}\cdots\\
&\ \ \  \bigcup_{k_{a_{3}}=1}^{\varphi(a_{3})}\cdots \bigcup_{k_{a_{n}-1}=1}^{L}\bigcup_{k_{a_{n}}=1}^{\varphi (a_{n})}E\begin{pmatrix}
1 & 2 & \ldots & 2n+1 \\ k_{1} & k_{2} & \ldots & k_{2n+1}
\end{pmatrix} \\
&\ \cup \bigcup_{k_{1}=1}^{\varphi (1)} I(k_{1}) \\
&\ \cup \bigcup_{k_{1}=1}^{\varphi (1)}\bigcup_{k_{2}=1}^{L}\cdots\bigcup_{k_{a_{2}-1}=1}^{L}\bigcup_{k_{a_{2}}=1}^{\varphi (a_{2})} I(k_{1},k_{2},\ldots ,k_{a_{2}}) \\
&\ \cup \bigcup_{k_{1}=1}^{\varphi (1)}\bigcup_{k_{2}=1}^{L}\cdots\bigcup_{k_{a_{2}}=1}^{\varphi (a_{2})}\cdots\bigcup_{k_{a_{3}-1}=1}^{L}\bigcup_{k_{a_{3}}=1}^{\varphi (a_{3})} I(k_{1},k_{2},\ldots ,k_{a_{3}}) \\
&\ \vdots \\
&\ \cup \bigcup_{k_{1}=1}^{\varphi (1)}\bigcup_{k_{2}=1}^{L}\cdots\bigcup_{k_{a_{2}}=1}^{\varphi (a_{2})}\cdots\\
&\ \ \ \bigcup_{k_{a_{3}}=1}^{\varphi (a_{3})}\cdots\bigcup_{k_{a_{4}}=1}^{\varphi (a_{4})}\cdots\bigcup_{k_{a_{n-1}-1}=1}^{L}\bigcup_{k_{a_{n-1}}=1}^{\varphi (a_{n-1})} I(k_{1},k_{2},\ldots ,k_{a_{n-1}}) .
\end{align*}

Let $F_{\varphi}^{(a_{n},L)}$ be the above set containing $F_{\varphi}^{(a_{n})}$. It follows from the definition that $F_{\varphi}^{(a_{n},L)}$ is a union of at most
\begin{align*}
& \sum_{j=1}^{n}(n-j+1)\varphi (a_{j}) + \sum_{t=1}^{n-1}\sum_{j=1}^{t}\big(a_{j+1}-a_{j} -1 \big)L \\ = & \sum_{j=1}^{n}(n-j+1)\varphi (a_{j}) + \left(\sum_{t=1}^{n}a_{t}-\frac{n(n+1)}{2}\right)L
\end{align*}
disjoint intervals and $F_{\varphi}^{(a_{n})}\subseteq F_{\varphi}^{(a_{n},L)}$.

We only need to check the last condition concerning the upper bound for $\mu\left(F_{\varphi}^{(a_{n},L)}\right)$. At first, observe that
\begin{align*}
\mu \big(I(k_{1},\ldots ,k_{N-1})\big) & =\left|\frac{p_{N-1}}{q_{N-1}} - \frac{p_{N-1}+\frac{p_{N-2}}{L+1}}{q_{N-1}+\frac{q_{N-2}}{L+1}}\right|\\ 
& =\frac{1}{L+1} \cdot \frac{1}{q_{N-1}\left(q_{N-1}+\frac{q_{N-2}}{L+1}\right)}<\frac{1}{L}\cdot \frac{1}{q_{N-1}^{2}},
\end{align*}
where $p_{j}/q_{j}$ are convergents of the continued fraction $[k_{1},\ldots ,k_{2N-1}]$. 

\medskip

The next two lemmas provide an upper bound for a multiple infinite power series containing fractions of the form $1/q_{N}^{2}$.

Note that results such as our Lemma \ref{LemProp3} appear, e.g., in \cite[Theorem 11]{Go}, \cite[page 168]{Cu}, however, we were not able to find a proof in the literature (and it is unclear if the cited statements are actually correct), and so we give a full proof.

\begin{lem}\label{LemProp3}
For a sequence $(k_{1},\ldots ,k_{s})$ of natural numbers,  denote the $j$th convergent of  $[k_{1};\ldots ,k_{s}]$ by $p_{j}(k_{1},\ldots ,k_{s})/q_{j}(k_{1},\ldots ,k_{s})$. Then for every $N\geq 1$:
\begin{align*}
\sum_{k_{1}=1}^{\infty}\sum_{k_{2}=1}^{\infty}\cdots \sum_{k_{N}=1}^{\infty}\frac{1}{q_{N}(k_{1},\ldots ,k_{N})^{2}}< 2.
\end{align*}
\end{lem}

Note that with slightly more effort, one can actually also show this result with $3/2$ as the upper bound.

\begin{proof}
From Lemma \ref{LemProp1} and the fact that the intervals of a given rank $N$ are pairwise disjoint and sum to the unit interval, we have
\begin{align*}
1 & =\sum_{k_{1}=1}^{\infty}\sum_{k_{2}=1}^{\infty}\cdots \sum_{k_{N}=1}^{\infty}\left| E\begin{pmatrix}
1 & 2 & \ldots & N \\ k_{1} & k_{2} & \ldots & k_{N}
\end{pmatrix} \right| \\
& = \sum_{k_{1}=1}^{\infty}
\cdots \sum_{k_{N}=1}^{\infty}\left| \frac{p_{N}(k_{1},\ldots ,k_{N})}{q_{N}(k_{1},\ldots ,k_{N})} - \frac{p_{N}(k_{1},\ldots ,k_{N}) + p_{N-1}(k_{1},\ldots ,k_{N-1})}{q_{N}(k_{1},\ldots ,k_{N}) + q_{N-1}(k_{1},\ldots ,k_{N-1})} \right| \\
& = \sum_{k_{1}=1}^{\infty}\sum_{k_{2}=1}^{\infty}
\cdots \sum_{k_{N}=1}^{\infty}\frac{1}{q_{N}(k_{1},\ldots ,k_{N}) \big(q_{N}(k_{1},\ldots ,k_{N}) + q_{N-1}(k_{1},\ldots ,k_{N-1})\big)} \\
& > \frac{1}{2} \sum_{k_{1}=1}^{\infty}\sum_{k_{2}=1}^{\infty}\cdots \sum_{k_{N}=1}^{\infty}\frac{1}{q_{N}(k_{1},\ldots ,k_{N})^{2}}
\end{align*}
The result follows.
\end{proof}

We now use Lemma \ref{LemProp3} to find an upper bound for $\mu\left(F_{\varphi}^{(a_{n},L)}\right)$:
\begin{align*}
\mu&\left( F_{\varphi}^{(a_{n},L)}\right)\\ 
\leq &\ \mu\left(\bigcup_{k_{1}=1}^{\varphi (1)}\bigcup_{k_{2}=1}^{L}\bigcup_{k_{3}=1}^{L}\cdots\bigcup_{k_{a_{2}}=1}^{\varphi (a_{2})}\cdots \bigcup_{k_{a_{n}-1}=1}^{L}\bigcup_{k_{a_{n}}=1}^{\varphi (a_{n})}E\begin{pmatrix}
1 & 2 & \ldots & a_{n} \\ k_{1} & k_{2} & \ldots & k_{a_{n}}
\end{pmatrix}\right) \\
&\ + \sum_{k_{1}=1}^{\varphi (1)} \mu (I(k_{1})) \\
&\ + \sum_{k_{1}=1}^{\varphi (1)}\sum_{k_{2}=1}^{L}\cdots\sum_{k_{a_{2}-1}=1}^{L}\sum_{k_{a_{2}}=1}^{\varphi (a_{2})} \mu(I(k_{1},\ldots ,k_{a_{2}})) \\
&\ \vdots \\
&\ + \sum_{k_{1}=1}^{\varphi (1)}\sum_{k_{2}=1}^{L}\cdots \sum_{k_{a_{n-1}-1}=1}^{L}\sum_{k_{a_{n-1}}=1}^{\varphi (a_{n-1})} \mu(I(k_{1},k_{2},\ldots ,k_{a_{n-1}})) \\
< &\ \mu \left(F_{\varphi}^{(a_{n})}\right)+\sum_{k_{1}=1}^{\varphi (1)}\frac{1}{L}\cdot\frac{1}{q_{1}(k_{1})^{2}}+\cdots\\ 
&\ +\sum_{k_{1}=1}^{\varphi (1)}\sum_{k_{2}=1}^{L}\cdots \sum_{k_{a_{n-1}}=1}^{\varphi (a_{n-1})}\frac{1}{L}\cdot\frac{1}{q_{a_{n-1}}(k_{1},\ldots ,k_{a_{n-1}})^{2}} \\
< &\ \mu \left(F_{B}^{(2n+1)}\right)+\frac{2(n-1)}{L}.
\end{align*}
Theorem \ref{StepII} follows.
\end{proof}

\subsection*{Step III} 
Now we want to prove the following quantitative version of \eqref{EquiLim}.

\begin{thm}\label{StepIII}
For every $X\in\mathbb{N}$ and for every interval $[a,b]\subseteq [0,1]$ we have
\begin{align*}
\left|\ \#\big\{\ 1\leq D\leq X\ \big|\  f_{D}\bmod{1} \in [a,b] \ \big\} - (b-a)X \ \right|\\
<(3\pi +1) X^{1/2}f_{X+1}^{1/2} + \frac{\pi}{2}\cdot\frac{1}{\Delta f_{X+1}}.
\end{align*}
\end{thm}
\begin{proof}
Let us denote the quantity on the left-hand side of the statement by $D(X,[a,b])$. We will use the Erd\H{o}s--Tur\'{a}n inequality, see \cite[Corollary 1.1]{Montgomery}.

\begin{lem}\label{LemTrig1}
For every positive integers $X$ and $K$,
\begin{align*}
\big| D(X,[a,b]) \big|\leq \frac{X}{K+1}+3\sum_{k=1}^{K}\frac{1}{k}\left|\sum_{D=1}^{X}e^{2\pi i k f_{D}}\right|.
\end{align*}
\end{lem}

Our task is to find a bound for the trigonometric sum $\left|\sum_{D=1}^{X}e^{2\pi i k f_{D}}\right|$. First, the following lemma is established within the proof of \cite[Theorem 2.5]{KuipersNiederreiters}.

\begin{lem}\label{LemTrig2}
Let $(f_{D})_{D=1}^{\infty}$ be a sequence of real numbers such that the sequence $\Delta f_{D}:=f_{D+1}-f_{D}$ is monotone. Then for every positive integers $X$ and $k$,
\begin{align*}
& \left|2\pi i k\sum_{D=1}^{X}e^{2\pi i k f_{D}}\right| \\ 
&\ \ \ \ \leq \sum_{D=1}^{X}\left|\frac{1}{\Delta f_{D}}-\frac{1}{\Delta f_{D+1}}\right|+2\pi^{2}k^{2}\sum_{D=1}^{X}\left|\Delta f_{D}\right|+\frac{1}{\left|\Delta f_{D+1}\right|}+\frac{1}{\left|\Delta f_{1}\right|}.
\end{align*}
\end{lem}

We assume that the sequence $(f_{D})_{D=1}^{\infty}$ is increasing and $(\Delta f_{D})_{D=1}^{\infty}$ is decreasing. Hence, Lemma \ref{LemTrig2} implies that for every $k$ we have
\begin{align*}
\left|2\pi i k\sum_{D=1}^{X}e^{2\pi i k f_{D}}\right|\leq &\ \frac{1}{\Delta f_{X+1}}-\frac{1}{\Delta f_{1}} + 2\pi^{2}k^{2}(f_{X+1}-f_{1})+\frac{1}{\Delta f_{X+1}}+\frac{1}{\Delta f_{1}}\\
 <&\  2\pi^{2}k^{2}f_{X+1} +\frac{2}{\Delta f_{X+1}}.
\end{align*}
Hence
\begin{align*}
\frac{1}{k}\left|\sum_{D=1}^{X}e^{2\pi i k f_{D}}\right| \leq \frac{1}{2\pi k^{2}}\left(2\pi^{2}k^{2}f_{X+1} +\frac{2}{\Delta f_{X+1}}\right)=\pi f_{X+1} + \frac{1}{k^{2}}\cdot\frac{1}{\pi \Delta f_{X+1}}.
\end{align*}

Let us plug in the above inequality into Lemma \ref{LemTrig1},
\begin{align*}
\big| D(X,[a,b]) \big| & <\frac{X}{K+1}+3\sum_{k=1}^{K}\left(\pi f_{X+1} + \frac{1}{k^{2}}\cdot\frac{1}{\pi \Delta f_{X+1}}\right)\\
& = \frac{X}{K+1}+3\pi f_{X+1} K + \frac{3}{\pi \Delta f_{X+1}}\sum_{k=1}^{K}\frac{1}{k^{2}} \\
 & < \frac{X}{K+1} + 3\pi f_{X+1} K +\frac{\pi}{2}\cdot\frac{1}{\Delta f_{X+1}}.
\end{align*}
The above inequality holds for every positive integer $K$. In particular, we can take $K=\lfloor X^{1/2}/f_{X+1}^{1/2}\rfloor$. Then
\begin{align*}
\big| D(X,[a,b]) \big|<\frac{X}{\lfloor X^{1/2}/f_{X+1}^{1/2}\rfloor+1}+3\pi \lfloor X^{1/2}/f_{X+1}^{1/2}\rfloor\sqrt{X} + \frac{\pi}{2}\cdot\frac{1}{\Delta f_{X+1}} \\
< (3\pi +1) X^{1/2}f_{X+1}^{1/2} + \frac{\pi}{2}\cdot\frac{1}{\Delta f_{X+1}},
\end{align*}
finishing the proof of Theorem \ref{StepIII}.
\end{proof}

\subsection*{Step IV} We are ready to prove the main statement.

\begin{proof}[Proof of Theorem $\ref{ThmBound}$]
Let us fix  numbers $n$ and $L$ and let $F_{\varphi}^{(a_{n},L)}$ be the set from Theorem \ref{StepII}. We have
\begin{align*}
&\frac{\#\big\{\ 1\leq D\leq X\ \big|\  f_{D}\bmod{1} \in F_{\varphi} \ \big\}}{X}\\ 
 \leq\  & \frac{\#\big\{\ 1\leq D\leq X\ \big|\  f_{D}\bmod{1} \in F_{\varphi}^{(a_{n},L)} \ \big\}}{X}.
\end{align*}

Let $I_{1}$, \ldots, $I_{R}$ be disjoint intervals such that $F_{\varphi}^{(a_{n},L)}=\bigcup_{j=1}^{R}I_{j}$. From Theorem \ref{StepIII} we know that for every~$j$,
\begin{align*}
\frac{\#\big\{\ 1\leq D\leq X\ \big|\  f_{D}\bmod{1} \in I_{j} \ \big\}}{X}<\mu (I_{j}) + (3\pi +1)\frac{f_{X+1}^{1/2}}{X^{1/2}} + \frac{\pi}{2} \frac{1}{X\Delta f_{X+1}}.
\end{align*}
By summing all such inequalities over $j$ we get
\begin{align*}
&\frac{\#\big\{\ 1\leq D\leq X\ \big|\  f_{D}\bmod{1} \in F_{\varphi}^{(a_{n},L)} \ \big\}}{X}\\ 
&\ \ \ \leq \mu\left(F_{\varphi}^{(a_{n},L)}\right)+\left((3\pi +1)\frac{f_{X+1}^{1/2}}{X^{1/2}} + \frac{\pi}{2}\cdot \frac{1}{X\Delta f_{X+1}}\right)R.
\end{align*}
The latter expression is bounded from above by
\begin{align*}
\prod_{j=1}^{n} & \left(1-\frac{1}{3(\varphi (j)+2)}\right)  +\frac{2(n-1)}{L} \\ 
& + \left((3\pi +1)\frac{f_{X+1}^{1/2}}{X^{1/2}}  + \frac{\pi}{2}\cdot \frac{1}{X\Delta f_{X+1}}\right)\\
&\cdot \left[ \sum_{j=1}^{n}(n-j+1)\varphi (a_{j}) + \left(\sum_{t=1}^{n}a_{t}-\frac{n(n+1)}{2}\right)L\right].
\end{align*}
The result follows.
\end{proof}

We can specialize Theorem \ref{ThmBound} to the case of the sequence $(\xi_{D})_{4\nmid D}$ with odd coefficients bounded by a fixed constant. Recall that
\begin{align*}
\xi_{D}:=\begin{cases}\sqrt{D} & \text{when }D \eq 2,3 \pmod 4, \\
\frac{1+\sqrt{D}}{2} & \text{when }D \eq 1 \pmod 4. \end{cases}
\end{align*}

\begin{cor}\label{CorBound}
For every  $X, B\geq 2$ satisfying $X\geq B^{12}(\log X)^{4}$, we have
\begin{align*}
\# \big\{\ 1\leq D\leq X \ \big|\  \xi_{D}\bmod 1 = [0;u_{1},u_{2},\ldots ],\ u_{2n-1}\leq B \textrm{ for all } n \ \big\}\\
 < 100 B^{3/2} (\log X)^{3/2} X^{7/8}.
\end{align*}
\end{cor}
\begin{proof}
We have
\begin{align*}
& \#  \big\{\ 1\leq D\leq X \ \big|\  \xi_{D}\bmod 1 = [0;u_{1},u_{2},\ldots ],\ u_{2n-1}\leq B \textrm{ for all } n \ \big\} \\
 & \leq \# \big\{\ 1\leq D\leq X \ \big|\  \sqrt{D}\bmod 1 = [0;u_{1},u_{2},\ldots ],\ u_{2n-1}\leq B \textrm{ for all } n \ \big\} \\
 & + \# \big\{\ 1\leq D\leq X \ \big|\  \frac{1+\sqrt{D}}{2}\bmod 1 = [0;u_{1},u_{2},\ldots ],\ u_{2n-1}\leq B \textrm{ for all } n \ \big\}.
\end{align*}
We bound the two summands above separately. The idea is simple: we just use Theorem \ref{ThmBound} twice. Once with $f_{D}:=\sqrt{D}$ and then with $f_{D}:=\frac{1+\sqrt{D}}{2}$. 

At first, let us use Theorem \ref{ThmBound} with $f_{D}:=\sqrt{D}$ for all $D$, and
\begin{align*}
\varphi (n)=\left\{\begin{array}{ll}
B, & 2\nmid n, \\
\infty, & 2\mid n.
\end{array}\right.
\end{align*}

Note that we indeed can take $f_{D}=\sqrt{D}$. Of course, it is increasing. Moreover,
\begin{align*}
\Delta f_{D}=\sqrt{D+1}-\sqrt{D}=\frac{1}{\sqrt{D+1}+\sqrt{D}}>\frac{1}{\sqrt{D+2}+\sqrt{D+1}}\\
=\sqrt{D+2}-\sqrt{D+1}=\Delta f_{D+1},
\end{align*}
so the sequence $(\Delta f_{D})_{D=1}^{\infty}$ is decreasing.

Since $X\geq 1$, we get
\begin{align*}
(3\pi +1)X^{1/2}f_{X+1}^{1/2}  +\frac{\pi}{2}\cdot\frac{1}{f_{X+2}-f_{X+1}} \\
 =(3\pi +1)X^{1/2}(X+1)^{1/4}+\frac{\pi}{2}\left((X+2)^{1/2}+(X+1)^{1/2}\right) 
& < 18X^{3/4}.
\end{align*}

For our choice of $f_{D}$ and $\varphi$ we also have $a_{n}=2n-1$ and $\varphi (a_{n})=B$ for all $n$. Hence, we get for all $X$, $L$ and $n$,
\begin{align*}
&\#  \big\{\ 1\leq D\leq X \ \big|\  \sqrt{D}\bmod{1} = [0;u_{1},u_{2},\ldots ],\ u_{2n-1}\leq B \textrm{ for all } n \ \big\} \\  & < X\left[\left(1-\frac{1}{3(B+2)}\right)^{n}+\frac{2(n-1)}{L}\right] + 18 X^{3/4}\left[\frac{n(n+1)}{2}B+\frac{n(n-1)}{2}L\right].
\end{align*}

By performing analogous computations for $f_{D}:=\frac{1+\sqrt{D}}{2}$ and the same choice of $\varphi$ as above we get
\begin{align*}
& \# \bigg\{\ 1\leq D\leq X \ \bigg|\  \frac{1+\sqrt{D}}{2}\bmod{1} = [0;u_{1},u_{2},\ldots ],\ u_{2n-1}\leq B \textrm{ for all } n \ \bigg\} \\  
& < X\left[\left(1-\frac{1}{3(B+2)}\right)^{n}+\frac{2(n-1)}{L}\right] + 22 X^{3/4}\left[\frac{n(n+1)}{2}B+\frac{n(n-1)}{2}L\right].
\end{align*}

Therefore, the quantity from the statement is bounded from above by
\begin{align}\label{IneqBound1}
2X\left[\left(1-\frac{1}{3(B+2)}\right)^{n}+\frac{2(n-1)}{L}\right] + 40 X^{3/4}\left[\frac{n(n+1)}{2}B+\frac{n(n-1)}{2}L\right].
\end{align}

Using inequalities $1-x\leq e^{-x}$, $B\geq 2$ and $B\leq L$ (we assume for a while that the last one is satisfied), we get that the last expression is further bounded by
\begin{align*}
2Xe^{-\frac{n}{6B}}+\frac{4(n-1)}{L}X+40X^{3/4}n^{2}L.
\end{align*}

Let us choose $n=\left\lceil \frac{3}{4}B\log X\right\rceil$ and $L=X^{1/8}\big(B\log X\big)^{-1/2}$. Then the latter expression is less than
\begin{align*}
2X^{7/8}+3B^{3/2}(\log X)^{3/2}X^{7/8}+40\left(\frac{3}{2}\right)^{2} B^{3/2} (\log X)^{3/2}X^{7/8}\\
 < 100 B^{3/2}(\log X)^{3/2}X^{7/8}.
\end{align*}

For the end of the proof observe that the assumption $L\geq B$ is equivalent to $X\geq B^{12}(\log X)^{4}$.
\end{proof}

One can also similarly establish the following version of our bound that is used by Man in his work on multiquadratic fields \cite{Man}.

\begin{cor}
For every $X, B\geq 2$  satisfying $X > B^{4} (\log X)^{4}$, we have
\begin{align*}
\# \big\{\ 1\leq D\leq X \ \big|\  \xi_{D}\bmod 1 = [0;u_{1},u_{2},\ldots ],\ u_{2n-1}\leq B \textrm{ for all } n \ \big\}\\
 < 50 B^{3/2} (\log X)^{3/2} X^{7/8} + 23 B^{3} (\log X)^{2} X^{3/4}.
\end{align*}
\end{cor}
\begin{proof}
The proof is analogous to the proof of Corollary \ref{CorBound}. We get the bound \eqref{IneqBound1} and again use the inequalities $1-x\leq e^{-x}$ and $B\geq 2$, but do not assume that $B\leq L$. Therefore, the quantity from the statement is bounded from above by
\begin{align*}
2X e^{-\frac{n}{6B}} + \frac{4(n-1)}{L}X + 20 X^{3/4}n^{2} L + 40 X^{3/4} n^{2} B.
\end{align*}
Again, we take $n=\left\lceil \frac{3}{4}B\log X\right\rceil$ and $L=X^{1/8}\big(B\log X\big)^{-1/2}$ and get further bounds:
\begin{align*}
2X^{7/8} & + 3 B^{1/2} (\log X)^{3/2} X^{7/8} + 45 B^{3/2} (\log X)^{3/2} X^{7/8}  + 23 B^{3} (\log X)^{2} X^{3/4} \\
& < 50 B^{3/2} (\log X)^{3/2} X^{7/8} + 23 B^{3} (\log X)^{2}X^{3/4}.
\end{align*}
The only condition we need to check is $L>1$. This is equivalent to $X > B^{4} (\log X)^{4}$ and hence the result follows.
\end{proof}

\section{Short vectors in $\Z $-lattices}\label{sec:theta est}

In this section, we will establish  
an upper bound for the number of vectors of a given norm in a quadratic $\Z$-lattice.
While such results are very well-known (e.g.  \cite[Lemma 4.1(b)]{Bl}, \cite[Theorem 20.9]{IwanKowal}), we provide a proof for completeness and since we could not find an appropriate explicit statement in the literature (however, a better statement than ours was proved by Regev and Stephens-Davidowitz \cite[Theorem 1.1]{RS} after our paper was finished; see below). 

First we need to introduce some definitions. Let $V$ be an $r $-dimensional vector space over $\Q$ equipped with a symmetric bilinear form $\mathfrak{B}:V\times V\rightarrow \Q$. Let $Q(v)=\mathfrak{B}(v,v)$ for  $v \in V$. A \textit{quadratic $\Z $-lattice} $\Lambda\subset  V$ is a $\Z $-submodule such that $\Q\Lambda = V$; $\Lambda$ is \textit{classical} if $\mathfrak{B}(v,w)\in\Z$ for all $v,w\in\Lambda$, and \textit{positive definite} if $Q(v)>0$ for all $0\neq v\in\Lambda$. 

For every quadratic form $Q$ over $\Q$ there is a corresponding quadratic $\Z $-lattice $(\Z^r, Q)$, where $r$ is the number of variables of $Q$. The associated \textit{Gram matrix} is defined as $G:=(\mathfrak{B}(e_i,e_j))$, where $e_i$ are the elementary vectors in $\R^r$ (or any other basis of the lattice). $\Gamma(x)$ denotes the usual \textit{gamma function}.

\begin{thm}\label{thm:short vectors}
	Let $\Lambda$ be a classical positive definite 
	$\Z $-lattice of rank $r$, $G$ the Gram matrix of a basis of $\Lambda$ and $n\geq 1$ be an integer. Let $N(n)$ denote the number of vectors of norm $n$ in $\Lambda$, i.e., of elements $v\in \Lambda$ such that $Q(v)=n$. Then
	\begin{align*}
	N(n) \leq C(r,n),	
	\end{align*}
where $C(r,n)$ is defined as follows:
\begin{align*}
	C(r,n):=\begin{cases}
		2r & \textrm{if } n=1, \\
		\max\{480, 2r(r-1)\} & \textrm{if } n=2, \\
		\frac{\pi^{\frac{r}{2}}}{\Gamma\left(\frac{r}{2}+1\right)} \frac{n^{\frac{r}{2}}}{\sqrt{\det G}} + \sum_{m=0}^{r-1}\binom{r}{m}\frac{\pi^{\frac{m}{2}}}{\Gamma\left(\frac{m}{2}+1\right)}n^{\frac{m}{2}} & \textrm{if } n\geq 3.
	\end{cases}
\end{align*}

Moreover, for all $r\geq 3$ and $n\ge 3$ we have
\begin{align*}
	C(r,n)\leq \frac{\pi^{\frac{r}{2}}}{\Gamma\left(\frac{r}{2}+1\right)} \frac{n^{\frac{r}{2}}}{\sqrt{\det G}} + \left(\frac{r\pi^{\frac{r-1}{2}}}{\Gamma\left(\frac{r+1}{2}\right)}+\frac{e^{330} (0.9)^{r}}{\sqrt{n}}\right) n^{\frac{r-1}{2}}.
\end{align*}

For further use, let us  define  $$B(r,n):=\frac{1}{2}C(2r,n).$$
\end{thm}

Alternatively, by the very recent result of Regev and Stephens-Davidowitz \cite[Theorem 1.1]{RS}, one can take 
$$C'(r,n):=2\binom{r+2n-1}{2n-1}-1\text{\ \ \  if \ \ \ }n\geq 3$$
for an upper bound, and then again define $B'(r,n):=\frac{1}{2}C'(2r,n).$

When $n$ is fixed and $r$ varies, then our bound is exponential in $r$, whereas the bound of Regev and Stephens-Davidowitz is much better, as it is polynomial in $r$. Conversely, when we fix $r$ and let $n$ vary, then our bound grows as $n^{r/2}$, whereas the bound of \cite{RS} grows faster, as $n^r$. Moreover, in the case of $\Lambda=\mathbb{Z}^r$, this bound almost matches the lower bound $\gg n^{\lfloor r/2\rfloor}$ that follows from inequality (4) in \cite{RS}.

\begin{proof}
The cases $n=1$ and $n=2$ are well known thanks to the characterization of lattices spanned by vectors of norm $1,2$  
\cite[Theorem 4.10.6, Proposition 4.10.7]{Ma} (for details of the argument that establishes these bounds, see, e.g., \cite[Section 7.1]{KT}).

We can move to the general case. 
Let
\begin{align*}
\mathcal{R} := \left\{\ \textbf{x}\in\mathbb{R}^{r}\ |\ Q(\textbf{x}) \leq n  \ \right\}.
\end{align*}

A theorem of Davenport \cite{Da}  implies that
\begin{align*}
N(n)\leq \Vol(\mathcal{R}) + \sum_{m=0}^{r-1}V_{m},
\end{align*}
where 
$\Vol (\mathcal{R})$ is the $r$-dimensional volume of $\mathcal{R}$ and 
$V_{m}$ is the sum of $m$-dimensional volumes of the projections of $\mathcal{R}$ on the various coordinate spaces obtained by equating any $r-m$ coordinates to $0$, and $V_{0}=1$ (in \cite[Theorem]{Da}, we take $h=1$, as the set $\mathcal R$ is convex).

Let $\lambda_{1},\ldots ,\lambda_{r}$ be the eigenvalues of the Gram matrix $G$ of a basis of $\Lambda$. Let us consider a linear transformation $\mathcal{A}:\mathbb{R}^{r}\to\mathbb{R}^{r}$ such that $G=\mathcal{A}^{T}J\mathcal{A}$, where $J$ is the Jordan form of $G$. In particular, $|\det\mathcal{A}|=1$. Therefore
\begin{align*}
\Vol\left(\mathcal{R}\right) & = \frac{1}{|\det\mathcal{A}|}\Vol\left(\mathcal{A}(\mathcal{R})\right)\\
& = \Vol\left(\left\{\ (y_{1},\ldots ,y_{r})\in\mathbb{R}^{r} \ \bigg| \ \frac{\lambda_{1}}{n}y_{1}^{2} + \cdots + \frac{\lambda_{r}}{n}y_{r}^{2}\leq 1 \right\}\right) \\
& = \frac{\pi^{\frac{r}{2}}}{\Gamma\left(\frac{r}{2}+1\right)}\prod_{j=1}^{r}\left(\frac{n}{\lambda_{j}}\right)^{\frac{1}{2}} = \frac{\pi^{\frac{r}{2}}}{\Gamma\left(\frac{r}{2}+1\right)} \frac{n^{\frac{r}{2}}}{\sqrt{\det G}}.
\end{align*}
In the above chain of equalities we used the well-known formula for the volume of hyperellipsoid.

One can perform analogous computation for every projection of $\mathcal{R}$ counted by $V_{m}$. Indeed, for each such projection we get a set defined by inequality $Q_{1}(\textbf{x})\leq n$ for some integral positive definite quadratic form $Q_{1}$ obtained from $Q$ by setting the values of some $r-m$ variables to be zero. Thus all the determinants of the corresponding Gram matrices are $\geq 1$ and we get for every $m$:
\begin{align*}
V_{m}\leq \binom{r}{m}\frac{\pi^{\frac{m}{2}}}{\Gamma\left(\frac{m}{2}+1\right)}n^{\frac{m}{2}}.
\end{align*}

Hence, we simply get
\begin{align*}
N(n) \leq \frac{\pi^{\frac{r}{2}}}{\Gamma\left(\frac{r}{2}+1\right)} \frac{n^{\frac{r}{2}}}{\sqrt{\det G}} + \sum_{m=0}^{r-1}\binom{r}{m}\frac{\pi^{\frac{m}{2}}}{\Gamma\left(\frac{m}{2}+1\right)}n^{\frac{m}{2}},
\end{align*}
as claimed.

\medskip

Let us now prove the `Moreover' part. We first write the previously obtained bound for $N(n)$ as
\begin{align*}
C(r,n) & =\frac{\pi^{\frac{r}{2}}}{\Gamma\left(\frac{r}{2}+1\right)} \frac{n^{\frac{r}{2}}}{\sqrt{\det G}}\\
& + \left(\frac{r\pi^{\frac{r-1}{2}}}{\Gamma\left(\frac{r+1}{2}\right)}+\frac{1}{\sqrt{n}}\sum_{m=0}^{r-2}\binom{r}{m}\frac{\pi^{\frac{m}{2}}}{\Gamma\left(\frac{m}{2}+1\right)}n^{-\frac{r-m-2}{2}}\right) n^{\frac{r-1}{2}}.
\end{align*}

We need to deal with the sum inside the brackets. We will use the following well-known bound:
\begin{align*}
s!>\left(\frac{s}{e}\right)^{s},
\end{align*}
which is true for all positive integers $s$. 

At first we consider the case of even numbers $m=2s\geq 98$:
\begin{align*}
\sum_{98\leq 2s\leq r-2}&\binom{r}{2s}\frac{\pi^{s}}{\Gamma\left(s+1\right)}n^{-\frac{r-2s-2}{2}} \\
& \leq 3^{1-\frac{r}{2}} \sum_{98\leq 2s\leq r-2}\binom{r}{2s}\frac{(3\pi)^{s}}{s!} < 3^{1-\frac{r}{2}} \sum_{98\leq 2s\leq r-2} \binom{r}{2s} \frac{(3\pi e)^{s}}{s^{s}} \\
& \leq 3^{1-\frac{r}{2}} \sum_{98\leq 2s\leq r-2} \binom{r}{2s} \frac{(3\pi e)^{s}}{49^{s}} =3^{1-\frac{r}{2}} \sum_{\substack{98\leq m\leq r-2 \\ 2\mid m}} \binom{r}{m}\left(\frac{3\pi e}{49}\right)^{\frac{m}{2}}.
\end{align*}

If $m=2s+1\geq 99$ we similarly get:
\begin{align*}
\sum_{99\leq 2s+1\leq r-2}&\binom{r}{2s+1}\frac{\pi^{s+\frac{1}{2}}}{\Gamma\left(s+\frac{3}{2}\right)}n^{-\frac{r-2s-3}{2}} \\
&  < 3^{1-\frac{r}{2}} \sum_{99\leq 2s+1\leq r-2}\binom{r}{2s+1}\frac{(3\pi)^{s+\frac{1}{2}}}{s!} \\
& < 3^{1-\frac{r}{2}} \sum_{99\leq 2s+1\leq r-2}\binom{r}{2s+1}\frac{(3\pi e)^{s+\frac{1}{2}}}{s^{s}} \\
& \leq 7\cdot 3^{1-\frac{r}{2}} \sum_{99\leq 2s+1\leq r-2}\binom{r}{2s+1}\frac{(3\pi e)^{s+\frac{1}{2}}}{49^{s+\frac{1}{2}}} \\
& = 7\cdot 3^{1-\frac{r}{2}} \sum_{\substack{99\leq m\leq r-2 \\ 2\nmid m}} \binom{r}{m}\left(\frac{3\pi e}{49}\right)^{\frac{m}{2}}.
\end{align*}
We can combine the above estimates and obtain:
\begin{align*}
\sum_{m=98}^{r-2}\binom{r}{m}\frac{\pi^{\frac{m}{2}}}{\Gamma\left(\frac{m}{2}+1\right)}n^{-\frac{r-m-2}{2}} < 7\cdot 3^{1-\frac{r}{2}} \sum_{m=98}^{r-2} \binom{r}{m} \left(\frac{3\pi e}{49}\right)^{\frac{m}{2}}\\
 < 21 \left(\frac{1}{\sqrt{3}}\left(\frac{3\pi e}{49} +1\right)\right)^{r} < 21 (0.9)^{r}.
\end{align*}

We bound the part corresponding to $m\leq 97$ in a different way. Let $e_{1}:=e^{1/e}$. One can check that for every $m$ the following inequality is true: $\frac{r^{m}}{e_{1}^{r}} \leq m^{m}$. Hence, 
\begin{align*}
\sum_{m=0}^{97}&\binom{r}{m}\frac{\pi^{\frac{m}{2}}}{\Gamma\left(\frac{m}{2}+1\right)}n^{-\frac{r-m-2}{2}} \\
& < 3^{1-\frac{r}{2}} \sum_{m=0}^{97}\frac{r^{m}}{m!}\frac{(3\pi)^{\frac{m}{2}}}{\Gamma\left(\frac{m}{2}+1\right)} = 3 \left(\frac{e_{1}}{\sqrt{3}}\right)^{r} \sum_{m=0}^{97}\frac{r^{m}}{e_{1}^{r}}\frac{(3\pi )^{\frac{m}{2}}}{m!} \\
& \leq 3 \left(\frac{e_{1}}{\sqrt{3}}\right)^{r} \sum_{m=0}^{97}\frac{(3\pi)^{\frac{m}{2}} m^{m}}{m!} < 3 \left(\frac{e_{1}}{\sqrt{3}}\right)^{r} \sum_{m=0}^{97}\frac{(3\pi)^{\frac{m}{2}} 97^{m}}{m!} \\
& < 3(0.9)^{r} e^{97\sqrt{3\pi}} < \left(e^{330} - 21\right) (0.9)^{r}.
\end{align*}
The result follows.
\end{proof}

\section{Large ranks of universal forms}\label{sec: large ranks}

Let us start this section by making more precise the definitions that we already used in the Introduction.
Recall that for fixed squarefree $D\in \Z_{>1}$ we denote
\begin{equation}\label{eq:gen}
\xi_D=\begin{cases}\sqrt{D} & \text{when }D \eq 2,3 \pmod 4, \\
\frac{1+\sqrt{D}}{2} & \text{when }D \eq 1 \pmod 4. \end{cases}
\end{equation}

The ring of integers of $H=\Q(\sqrt{D})$ we denote by $\co_H$. Every $\alpha \in \co_H$ can be written as $\alpha=a+b\xi_D$, with $a,b \in \Z$. We say that $\alpha$ is \textit{totally positive}, denoted $\alpha \succ 0$, if both $\alpha$ and its conjugate $\alpha'$ are greater than zero, i.e., $\alpha,\alpha'>0.$ Here $\alpha'=a+b\xi_D'$, where

\begin{equation*}
\xi_D'=\begin{cases}-\xi_D & \text{when }D \eq 2,3 \pmod 4, \\
1-\xi_D & \text{when }D \eq 1 \pmod 4. \end{cases}
\end{equation*}

 The sets of all totally positive numbers and integers are denoted by $H^+$ and $\co_H^+$, respectively. Let $\Tr_{H/\Q}, \Nr_{H/\Q}:H\rightarrow \Q$ denote the \textit{trace} and \textit{norm} maps. Explicitly, $\Tr_{H/\Q}(\alpha)=\alpha+\alpha'$ and $\Nr_{H/\Q}(\alpha)=\alpha\alpha'$. The \textit{codifferent} of $\co_H$ is $\co_H^{\vee}=\{\alpha \in H\mid \Tr_{H/\Q}(\alpha \co_H)\subset \Z\}$.
 
 \medskip
 
We briefly introduce the language of quadratic $\co_{H} $-lattices, which will be used below. Consider an $R $-dimensional vector space $V$ over $H=\Q(\sqrt{D})$ equipped with a symmetric bilinear form $\mathfrak{B}:V\times V\rightarrow H$. Let $Q(v)=\mathfrak{B}(v,v)$ for $v \in V$. A \textit{quadratic $\co_H $-lattice $\Lambda\subset  V$ of rank $R$} is an $\co_H $-submodule such that $H\Lambda = V$. 
When talking about quadratic $\co_H $-lattices, we always mean the pair $(\Lambda,Q)$. 

All our lattices are assumed to be \textit{integral} in the sense that $Q(v)\in\co_H$ for all $v\in\Lambda$.
We are primarily working with lattices that are \textit{classical} in the sense that, moreover, $\mathfrak{B}(v,w)\in\co_H$ for all $v,w\in\Lambda$. When this assumption is not satisfied, then we say that the lattice is \textit{non-classical}; in such a case, $\mathfrak{B}(v,w)\in\frac 12\co_H$.

We say that $(\Lambda,Q)$ is \textit{totally positive definite} if $Q(v)\succ 0$ for every non-zero $v\in L$. Given that we work exclusively with totally positive definite quadratic lattices, we shall refer to them just as quadratic $\co_H$-lattices. A lattice \textit{represents} $\alpha \in \co_H^+$ if  $Q(v)=\alpha$ for some $v\in\Lambda$; if it represents all elements of $\co_H^+$, then it is \textit{universal}. Furthermore, we shall say that a quadratic $\co_H$-lattice is \textit{$m\co_H $-universal} if it represents all elements of $m\co_H^+$, for some fixed $m \in \Z_{\geq 1}$. 

For some more relevant background and interesting results on quadratic lattices, see, e.g., \cite{BC+,CO, CS1, Ea,Ka3,KY1,KKO,Km,O}.

\medskip

Let again $D$ be a squarefree positive integer and 
$\xi_D$ as in \eqref{eq:gen}. 
Let $\xi_D=[u_0; \overline{u_1,\dots, u_s}]$ be the periodic continued fraction.

Recall that Kala and Tinkov\' a proved \cite[Section 7.3]{KT} that every universal lattice over $\co_H$ has rank $>u_{2i+1}/2$ (for any integer $i$). In Section \ref{sec:cont frac}, we established a density zero result for the set of $D$s such that all the coefficients of $\xi_D$  with odd indices are small. From this we will deduce that outside of this density-zero set, lattices that are universal, or just represent all of $m\co_H^+$, must have large ranks. 
For that, it suffices to modify the argument of \cite{KT} using the bound on the number of vectors of a given norm in a $\Z $-lattice from Theorem \ref{thm:short vectors}.

We are now ready to prove the first main theorem of this section.

\begin{thm}\label{thm:large ranks}
	Let $D,m$ be positive integers such that $D>1$ is squarefree, and $H=\Q(\sqrt D)$. Let $\xi_D=[u_0; \overline{u_1,\dots, u_s}]$ be the periodic continued fraction and let $u=\max\{u_{2i+1}|i\geq 0\}$. 
	
	Assume that $\Lambda$ is an $m\co_H$-universal classical quadratic $\co_H $-lattice of rank $R$. Then
	$$u< B(R,m),$$
where $B(R,m)$ is defined in Theorem $\ref{thm:short vectors}$.
\end{thm}	

\begin{proof} 
	Assume that $u=u_{2i+1}$ and consider the ``semiconvergents'' $B_r:=\alpha_{2i-1}+r\alpha_{2i}, 0\leq r\leq u_{2i+1}$.
	By \cite[Section 3]{KT}, each $B_r\succ 0$ and there is $\delta\in\co_H^{\vee, +}$ such that $\Tr(\delta B_r)=1$.
	
	By fixing a $\Z $-basis for $\co_H$, we can identify $\Lambda$ with a classical $\Z $-lattice of rank $2R$ via an isomorphism $\varphi:\Lambda\rightarrow \Z^{2R}$ and equip it with the quadratic form $q(v):=\Tr(\delta Q(\varphi^{-1}(v)))$. As $\delta\in\co_H^{\vee, +}$, the quadratic form $q$ on $\Z^{2R}$ is classical and positive definite.
	
	By our assumption, $Q$ represents all of $m\co_H^+$. In particular, there are vectors $w_r$ such that $Q(w_r)=mB_r$.
	For the corresponding vectors $v_r:=\varphi(w_r)$ we have $q(\pm v_r)=q(v_r)=\Tr(\delta Q(w_r))=\Tr(m\delta B_r)=m$.
	
	Thus the number $N(m)$ of vectors of norm $m$ in our lattice $(\Z^{2R},q)$ satisfies $N(m)\geq 2(u+1)>2u$.
	In order to conclude the proof, it is enough to combine this inequality with Theorem \ref{thm:short vectors}.
\end{proof}

Finally, we can prove the main results of the paper.

\begin{thm}\label{th:main}
	Let $R,m,X$ be positive integers. Denote
\begin{align*}
	&\mathcal{D}(R,m,X):=\\
	&\ \#\left\{ \text{squarefree }  D\leq X \mid \exists \text{ $m\co_{\Q(\sqrt{D})}$-universal classical lattice of rank } R  \right\}.
\end{align*}
	Let $B(R,m)$ be the quantity defined in Theorem $\ref{thm:short vectors}$. Then for all $X$ satisfying $X\geq B(R,m)^{12}(\log X)^{4}$ we have
	\begin{align*}
		\mathcal{D}(R,m,X)<100 B(R,m)^{3/2}\cdot X^{7/8}(\log X)^{3/2}.
	\end{align*}
\end{thm}

This establishes Theorem \ref{MAIN}
with $A(R,m):=100 B(R,m)^{3/2}$.

\begin{proof}
	Theorem \ref{thm:large ranks} implies that $u<B(R,m)$ and the result then follows from Corollary \ref{CorBound}.
\end{proof}

\begin{cor}\label{cor:main}
For every $\varepsilon >0$ and sufficiently large $X$, we have
\begin{align*}
&\#\left\{  \text{ squarefree }  D\leq X\ \bigg|\ R_{class}(\Q(\sqrt D))\leq D^{\frac{1}{12} -\varepsilon}  \right\} < 300X^{1-\frac{3}{2}\varepsilon}(\log X)^{3/2},\\
&\#\left\{  \text{ squarefree }  D\leq X\ \bigg|\ R(\Q(\sqrt D))\leq D^{\frac{1}{24} -\varepsilon}\right\}\ \ \ \ \  < 800X^{1-3\varepsilon}(\log X)^{3/2}.
\end{align*}
\end{cor}

This corollary obviously implies Theorem \ref{th1.1}.

\begin{proof}
For the first estimate, take $X, R$ such that 
$R=X^{1/12-\varepsilon}$.
By Theorem \ref{th:main}, if $X$ is sufficiently large (specifically, $X^{12\varepsilon}\geq (\log X)^{4}$), then the number of squarefree $D\leq X$ that admit a universal lattice of rank at most $R=X^{1/12-\varepsilon}\geq D^{1/12-\varepsilon}$ is less than $100\cdot 2^{3/2}X^{1-\frac{3}{2}\varepsilon}(\log X)^{3/2} < 300 X^{1-\frac{3}{2}\varepsilon}(\log X)^{3/2}$, as we want.

For the second estimate, take $X, R$ such that 
$R^2=X^{1/12-2\varepsilon}$.
Now if there is a (possibly) \textit{non-classical} universal lattice $(\Lambda, Q)$ over $\Q(\sqrt D)$, then $(\Lambda, 2Q)$ is a classical $2\co_{\Q(\sqrt{D})}$-universal lattice of the same rank. The result again immediately follows from the preceding theorem.
\end{proof}

\end{document}